\documentclass[preprint,12pt]{elsarticle}
\usepackage[english]{babel}
\usepackage[utf8]{inputenc}
\usepackage{amsmath,amssymb,amsfonts}
\usepackage{amsthm} 
\usepackage{textcomp,epsfig,mathtools}
\usepackage{xcolor}
\usepackage{comment}

\newcommand{\bbH}{\mathbb{H}}
\newcommand{\bbF}{\mathbb{F}}
\newcommand{\mc}[1]{\mathcal{#1}}
\newcommand{\bbR}{\mathbb{R}}
\newcommand{\bbZ}{\mathbb{Z}}
\newcommand{\bbN}{\mathbb{N}}
\newcommand{\N}{\operatorname{N}}
\newtheorem{theorem}{Theorem}

\newtheorem{cor}[theorem]{Corollary}
\newtheorem{prop}[theorem]{Proposition}
\newtheorem*{conj*}{Conjecture}

\theoremstyle{remark}
\newtheorem{rmk}[theorem]{Remark}

\makeatletter
\def\ps@pprintTitle{%
  \let\@oddhead\@empty
  \let\@evenhead\@empty
  \let\@oddfoot\@empty
  \let\@evenfoot\@oddfoot
}
\makeatother
\begin{document}

\begin{frontmatter}

  \title{On the Cycle Structure of the Metacommutation Map}

  \author{António Leite$^*$, António Machiavelo} 

  \cortext[]{Corresponding author: up201805230@up.pt}

  \address{Centro de Matemática da Universidade do Porto\\ 4169-007
    Porto, Portugal}

  \begin{abstract}
    Cohn and Kumar showed that the permutation on the set of the
    classes of left associated Hurwitz primes above an odd prime $p$
    induced through metacommutation by a Hurwitz prime $\xi$ of norm
    $q$ has either $0$, $1$ or $2$ fixed points, and that the
    permutation $\tau_{\xi,p}$ induced on the non-fixed points splits
    into cycles of the same length. Here we show how to find the
    length of those cycles, in terms of $p$ and $\xi$, using
    cyclotomic polynomials over $\mathbb{F}_p$. We then show that,
    given an odd prime $p$, there is always a prime quaternion $\xi$
    such that the permutation $\tau_{\xi,p}$ has only one non-trivial
    cycle of length $p$. Finally, we give conditions for a prime $\pi$
    of norm $p$ to be a fixed point of the aforementioned
    metacommutation map.
  \end{abstract}
\end{frontmatter}
\section{Introduction}

Consider the Hamilton quaternions,
$\bbH=\{a+bi+cj+dk\ |\ a,b,c,d \in \bbR\}$, the multiplication in
$\bbH$ being determined by $i^2=j^2=k^2=ijk=-1$.
In this paper, we study its subring of the so called Hurwitz
quaternions,
\begin{equation*}
  \mc H=\left\{\frac{1}{2}(a+bi+cj+dk)\in \bbH\ |\ a\equiv b \equiv c
    \equiv d \pmod{2} \right\}. 
\end{equation*}

Recall that, for any quaternion $\alpha=a+bi+cj+dk \in \bbH$, the
quaternion $\bar{\alpha}=a-bi-cj-dk$ is its conjugate. We denote its
real part by $\Re(\alpha)$ and its vector part by $\mc V(\alpha)$. Its
norm is $\operatorname{N}(\alpha)=\alpha\bar{\alpha}=a^2+b^2+c^2+d^2$,
and the trace is
$\operatorname{Tr}(\alpha)=\alpha+\bar{\alpha}=2 \Re(\alpha)$.  Since
the ring of Hurwitz quaternions is left and right Euclidean for the
norm, it is also a left and right PID. An element $\pi \in \mc H$ is
called prime if it is irreducible. It turns out that a Hurwitz integer
$\pi$ is prime in $\mc H$ if and only if $\operatorname{N}(\pi)=p$
with $p$ prime. In this case, we say that $\pi$ is a prime quaternion
\emph{above} the rational prime $p$.

As a consequence of the existence of the Euclidean algorithms, it can
be shown \cite[Lecture 10]{hurwitz2023} that any non-zero
$\gamma\in\mc H$ with norm
$\operatorname{N}(\gamma)=p_1p_2\cdots p_n$, where $p_i$ are prime,
can be factored as
$$\gamma= \pi_1 \pi_2\cdots \pi_n,\ \quad \text{ where }
\operatorname{N}(\pi_i)=p_i.$$
As in \cite{ConwaySmith}, we say that this factorization of $\gamma$
is modeled on the factorization $p_1 p_2\cdots p_n$ of
$\operatorname{N}(\gamma)$. When $\gamma$ is not a multiple of an
integer greater than $1$, this factorization is unique up to unit
migration, that is, every factorization of $\gamma$ modeled on the
factorization $\operatorname{N}(\gamma) = p_1 p_2\cdots p_n$ (where
order matters) is of the form
$$\gamma= (\pi_1u_1)(u_1^{-1}\pi_2u_2)\dots (u_{n-1}^{-1}\pi_n),\
\enspace \text{ where } \operatorname{N}(\pi_i)=p_i \text{ and } u_i
\text{ are units.} $$

Consider, in particular, the quaternion $\gamma=\pi\xi$, modeled on
$\operatorname{N}(\gamma)=pq$, where $p$ and $q$ are distinct rational
primes. Then $\gamma$ has a unique factorization $\xi'\pi'$ modeled on
$q p$, up to unit migration. This process of swapping adjacent primes
was called \textit{metacommutation} by Conway and Smith, in
\cite{ConwaySmith}. In fact, they proved that the prime factorization
of an arbitrary Hurwitz integer is unique up to metacommutation, unit
migration and recombination, the replacement in a factorization of
$\pi\bar{\pi}$ with $\pi_1\bar{\pi}_1$, where
$\operatorname{N}(\pi)=\operatorname{N}(\pi_1)$.

For every odd rational prime $p$, there are, up to left associates
(i.e.~up to left multiplication by a unit), precisely $p+1$ Hurwitz
primes above $p$, as first shown by Hurwitz
\cite[p.~94]{hurwitz2023}. Let $\Pi_p$ denote the set of these $p+1$
classes of left associated quaternion primes above $p$. By
metacommutation, a prime $\xi \in \mc H$ induces a permutation
$\tau_{\xi,p}$ on $\Pi_p$ by sending the class of $\pi$ to the class
of the prime $\pi'$ which satisfies $\pi\xi=\xi'\pi'$.  One calls this
map the metacommutation map induced on $\Pi_p$ by $\xi$.

In \cite{CohnKumar}, Cohn and Kumar prove the following theorem about
the metacommutation map that motivates the results in this paper.

\begin{theorem}[Cohn and Kumar, Theorem 1.1]
  \label{thm1}
  Let $p$ and $q$ be distinct rational primes, let $\xi$ be a Hurwitz
  prime of norm $q$. The metacommutation map $\tau_{\xi,p}$ is a
  permutation, whose sign is exactly $(\frac{q}{p})$.

  The map $\tau_{\xi,p}$ is the identity if either $p=2$ or if $\xi$
  is congruent to a rational integer modulo $p$. Otherwise, it has
  $1+\left(\frac{\operatorname{Tr}(\xi)^2-4q}{p}\right)$ fixed points.
\end{theorem}

Let $\mc H_p=\mc H/p\mc H$, a quaternion algebra over $\bbF_p$, which
has dimension $4$ (see \cite[Math Theorem 5.4.4]{Voight}). For any
quaternion $\xi$, we denote its reduction modulo $p$ by $\xi_p$. As it
is well known, the ideals of $\mc H_p$ are the image of the ideals of
$\mc H$ that contain $p$ by the natural projection map
$\mc H\to\mc H_p$.  Since $\mc H$ is a PID, the ideals that contain
$p$ are ideals generated by divisors of $p$, and thus the non-trivial
ideals of $\mc H_p$ are the ideals $\mc H_p \pi_p$, where $\pi$ is a
prime quaternion above $p$. Let $\mc I_p$ be the set of these
ideals. It follows that there exists a one-to-one correspondence
between $\Pi_p$ and $\mc I_p$.

In \cite[Lemma 2.1]{CohnKumar} it is remarked that there is in
$\mc H_p\pi_p$, up to scalar multiplication, an unique element of
trace zero, that we will here denote by $t_\pi$. Moreover, since any
element of $\mc H\pi$ has norm divisible by $p$, it follows that one
can naturally associate to $t_\pi = xi+yj+zk$ ($x,y,z\in\bbF_p$) the
point $c_\pi:=[x:y:z]$ in the conic
\begin{equation*}
C_p=\left\{[x: y: z] \in \mathbb{P}^2\left(\mathbb{F}_p\right)
    \mid x^2+y^2+z^2=0\right\}.
\end{equation*}
In fact, the map $\mc I_p\to C_p$ given by
$\mc H_p \pi_p\mapsto c_\pi$ is a bijection, since given any point
$c=[x:y:z]\in C_p$, if we consider the element
$t=x i+y j+ zk\in\mc{H}_p$, then it is easy to see that the ideal
$\mc H_p t$ is non-trivial, hence of the form $\mc H_p \pi_p$, for
some prime $\pi$ above $p$. We, thus, get the bijection
\begin{center}
  \begin{tabular}{ccc}
    $\Pi_p$ & $\longrightarrow$ & $C_p$ \\
    $[\pi]$ & $\mapsto$ & $c_\pi$
  \end{tabular}
\end{center}
of Proposition 2.2 in \cite{CohnKumar}.

When $\pi \xi=\xi'\pi'$, we have, by \cite[Theorem 3.1]{CohnKumar},
that the unique element $t_{\pi'}$, up to scaling, is equal to
$(\xi_p)^{-1}t_\pi\xi_p$. This means that, if $t_\pi=xi+yj+zk$, we
have, up to scaling, that
$t_{\pi'}=(\xi_p)^{-1}t_\pi\xi_p=x'i+y'j+z'k$, with
$(x',y',z')=\phi_{\xi,p}(x,y,z)$, where $\phi_{\xi,p}$ is the linear
endomorphism of $\bbF_p^3$ given by the matrix
$$\frac{1}{q}\left(
  \begin{array}{ccc}
    a^2+b^2-c^2-d^2 & 2 a d+2 b c & -2 a c+2 b d \\
    -2 a d+2 b c & a^2-b^2+c^2-d^2 & 2 a b+2 c d \\
    2 a c+2 b d & -2 a b+2 c d & a^2-b^2-c^2+d^2
  \end{array}\right),
$$
and $\xi_p = a+bi+cj+dk$.  The characteristic polynomial of
$\phi_{\xi,p}$ is $(x-1)f_{\xi,p}(x)$, where
$f_{\xi,p}(x)=x^2+\left(2-\frac{\operatorname{Tr}(\xi)^2}{q}\right)
x+1\in\bbF_p[x]$ (see~\cite[Lemma 3.2]{CohnKumar}).

The proof of Theorem \ref{thm1} relies on the fact that, following the
bijections described above, we can think of $\tau_{\xi,p}$ as the left
action of $\phi_{\xi, p}$ on the points $C_p$. This means that, if we
view $\phi_{\xi, p}$ as a projective transformation, the number of
fixed points of $\tau_{\xi,p}$ is simply the number of fixed points of
$\phi_{\xi,p}$ that are in $C_p$, or equivalently, the number of
eigenvectors of $\phi_{\xi,p}$, viewed as a transformation in
$\bbF_p^3$, that lie on the set
$\{(x,y,z)\in \bbF_p^3 \mid x^2+y^2+z^2=0\}$.  Additionaly, the sign
of the permutation is found using the fact that the group $SO(g_t)$,
where $g_t$ is the diagonal form $x^2-ty^2$ with
$t\in \bbF_p\setminus\{0\}$, is cyclic and its elements are semisimple
(see \cite[Proposition 5.3]{CohnKumar}). In fact, if $\tau_{\xi,p}$ is
neither the identity, nor a cycle of length $p$, it is proved in
\cite[Theorem 6.1]{CohnKumar} that, in order to understand the orbits
of $\phi_{\xi,p}$ on $C_p$, it is enough to understand the action of a
special element $\psi_{\xi,p}\in SO(g_t)$, whose characteristic
polynomial is precisely $f_{\xi,p}$, on the vectors of a fixed norm
$u$. In the general case, this is the permutation induced by
$\psi_{\xi,p}$ on the affine conic $D:=x^2-ty^2=u$, whose points are
precisely those with norm $u$. It happens that, since $SO(g_t)$ is
cyclic and its action on $D$ is simply transitive (which is proved in
\cite[Lemma 5.5]{CohnKumar}), the permutation induced by
$\psi_{\xi,p}$ is a product of
$\frac{\lvert SO(g_t) \rvert}{\operatorname{ord}(\psi_{\xi,p})} $
disjoint $\operatorname{ord}(\psi_{\xi,p})$-cycles and its sign is
precisely $(\frac{q}{p})$. This proves that all the cycles which are
not fixed points have the same length. This result, which is also
proved in Theorem 5.3 of \cite{ForsythEtAll}, implies that the cycle
length of the non-fixed points, which we denote $\ell_{\xi,p}$,
divides $p+1$, $p$ or $p-1$ depending on whether the permutation has
0, 1 or 2 fixed points, respectively.  We call cycles of length
greater than one \emph{non-trivial}. Note that, by Theorem \ref{thm1},
assuming the existence of non-trivial cycles is the same as saying
that $p$ must be odd and $\xi$ must not be congruent to an integer
modulo $p$.

In his PhD thesis, \cite{Nikos}, N. Tsopanidis gives criteria to
decide when the non-trivial cycles are of length either $2$ or $3$. In
this paper, we generalize Tsopanidis' results to all lengths, by
noticing that $\ell_{\xi,p}=t$, with $t\in \mathbb{N}$ with $t>1$, if
and only if the roots of $f_{\xi,p}$ are primitive $t^{\rm\, th}$
roots of unity.

\section{Cycle Structure}

Given a prime quaternion $\xi$, an odd rational prime $p$, and
$t \in\bbN$, with $t>1$, the following theorem provides a necessary
and sufficient condition for the non-trivial cycles of $\tau_{\xi,p}$
to have length $t$.

\begin{theorem}\label{mainthm}
  Let $p$ be an odd prime, $\xi$ a Hurwitz prime not congruent to a
  rational integer modulo $p$ with reduced norm $q$, and let
  $\ell_{\xi,p}$ be the length of the non-trivial cycles of
  $\tau_{\xi,p}$. Then $\ell_{\xi,p}=t$ if and only if
  $\left(f_{\xi,p}(x), \Phi_t(x)\right)\neq 1$, where $f_{\xi,p}$ is the
  polynomial
  $f_{\xi,p}(x)=x^2+\left(2-\frac{\operatorname{Tr}(\xi)^2}{q}\right)x+1$
  and $\Phi_t(x)$ is the $t^{\rm\, th}$ cyclotomic polynomial over
  $\mathbb{F}_p$.
\end{theorem}

\begin{proof}
  Let $\xi \equiv a+bi+cj+dk \pmod{p}$ and
  $\operatorname{N}(\xi)\equiv q\pmod{p}$. Since $p$ is odd, and $\xi$
  is not congruent to an integer modulo $p$, the metacommutation map
  $\tau_{\xi,p}$ is not the identity, hence it has non-trivial cycles.

  We know, from our earlier discussion, that we can think of
  $\tau_{\xi,p}$ as the left action of the matrix $\phi_{\xi,p}$ on
  the points of the conic $ C_p$. Note that in $C_p$ we have $3$
  linearly independent vectors. If the non-trivial cycles have length
  $\ell_{\xi,p}=t$, which has to be $p$ or a divisor of $p\pm1$,
  depending on the number of fixed points, it follows that
  $(\phi_{\xi,p}) ^t$ fixes every point of the conic $C_p$, and
  therefore it must be the identity operator. In other words, we must
  have $(\phi_{\xi,p})^t=I$ and $(\phi_{\xi,p})^d\neq I,$ for all
  $d \in \{1,\dots,t-1\}$.

  In the case $\ell_{\xi,p}=p$, the metacommutation map $\tau_{\xi,p}$
  has only one fixed point, thus it follows from Theorem 1 that
  $a^2\equiv q \pmod{p}$, which means that $f_{\xi,p}(x)=(x-1)^2$. The
  $p^{\rm\, th}$ cyclotomic polynomial over $\mathbb{F}_p$ is simply
  $(x-1)^{p-1}$, so $\left(f_{\xi,p}(x), \Phi_p(x)\right)\neq 1$. Now,
  if $(f_Q(x),\Phi_p(x))\neq1$, then $(x-1)$ must divide $f_{\xi,p}$
  and we get that $a^2=q$. Thus, there is only one fixed point and the
  remaining points form a cycle of length $p$, i.e., $\ell_{\xi,p}=p$.

  In the case $\ell_{\xi,p}\mid p\pm1$, $p$ does not divide
  $\ell_{\xi,p}=t$, and thus, from Theorem 2.45 in
  \cite{LidlNiederreiter}, the polynomial $x^t-1$ can be written as
  $\Pi_{d \mid t} \Phi_d(x)$, where $\Phi_d(x)$ is the $d^{\rm\, th}$
  cyclotomic polynomial over $\mathbb{F}_p$. By the Cayley-Hamilton
  theorem, we have that $\phi_{\xi,p}$ satisfies the equation
  $x^t-I=0$, where $I$ denotes the identity matrix. Hence, we must
  have $\Pi_{d \mid t} \Phi_d(\phi_{\xi,p})=0$. Therefore,
  $\Phi_t(\phi_{\xi,p})=0$, since $\Phi_d(\phi_{\xi,p})\neq 0$, for
  every $d < t$, as otherwise we would have $(\phi_{\xi,p})^d=I$. From
  the fact that the characteristic polynomial of $\phi_{\xi,p}$ is
  $(x-1)f_{\xi,p}(x)$, one must also have that $\phi_{\xi,p}$
  satisfies the equation $f_{\xi,p}(\phi_{\xi,p})=0$. Thus,
  $\Phi_t(\phi_{\xi,p})=0$ and $f_{\xi,p}(\phi_{\xi,p})=0$. If we
  suppose $\left(f_{\xi,p}(x),\Phi_t(x)\right)=1$, then it must exist
  polynomials $g$ and $h$ such that
  $g(x)f_{\xi,p}(x)+h(x)\Phi_t(x)=1$. In particular this implies that
  $g(\phi_{\xi,p})f_{\xi,p}(\phi_{\xi,p})+h(\phi_{\xi,p})\Phi_t(\phi_{\xi,p})=I$,
  and so $0=I$, which is a contradiction.

  For the reverse implication, let
  $\left(f_{\xi,p}, \Phi_t\right) \neq 1$. From this it follows that
  there exists a $\lambda$ in the algebraic closure of $\mathbb{F}_p$,
  i.e., $\lambda \in \bar{\mathbb{F}}_p$, such that
  $f_{\xi,p}(\lambda)=\Phi_t(\lambda)=0$. Therefore, since $\lambda$
  satisfies the $t^{\rm\, th}$ cyclotomic polynomial, $\lambda^t=1$
  and $\lambda^d \neq 1$ for all $d < t$, which implies that its order
  is $t$. Since $f_{\xi,p}$ is the characteristic polynomial of a
  specific element $\psi_{\xi,p} \in SO_2(g_t)$, we know, from
  Proposition 5.3 in \cite{CohnKumar}, that $\psi_{\xi,p}$ is
  semisimple and its order is equal to the order of $\lambda$. Also,
  from our previous discussion, we know that we can reduce the study
  of the orbits of $\phi_{\xi,p}$ on $C_p$ to the study of the cyclic
  action of the corresponding $\psi_{\xi,p}$ on the affine conic
  $D$. Hence, the order of $\phi_{\xi,p}$ must be the same as the
  order of $\psi_{\xi,p}$. Therefore, since $\psi_{\xi,p}$ has order
  $t$, the non-trivial cycles must all have order $t$, that is,
  $\ell_{\xi,p}=t$.
\end{proof}

\begin{rmk}
  We know that $f_{\xi,p}$ and $\Phi_t$ have a common divisor in
  $\mathbb{F}_p$ if and only if $R(f_{\xi,p},\Phi_t)\equiv0\pmod{p}$,
  where $R$ denotes the resultant between the two polynomials, and
  therefore $\ell_{\xi,p}=t$ if and only if, modulo $p$,
  $R(f_{\xi,p},\Phi_t)=0$.
\end{rmk}

From this Theorem one can extract criteria for when the length
$\ell_{\xi,p}$ is equal to some prescribed value. In particular, the
results N. Tsopanidis presented in his PhD thesis \cite{Nikos}, for
when $\ell_{\xi, p}$ is equal to $2$ or $3$, can be obtain from this
theorem.

\begin{cor}[Propositions 5.4 and 5.5 of \cite{Nikos}]
  The non-trivial cycles of $\tau_{\xi,p}$ have length $2$ if and only
  if $\xi$ is pure modulo $p$, and they have length $3$ if and only if
  $\N(\xi)\equiv \operatorname{Tr}(\xi)^2 \pmod{p}$.
\end{cor}

As an example, suppose that $\ell_{\xi,p}=2$. It follows that
$\left(f_{\xi,p}, \Phi_2\right)\neq 1$, or equivalently, that we must
have $R(f_{\xi,p},\Phi_2)=0$. Since $\Phi_2=x+1$ and the degree of
$f_{\xi,p}$ is two, our condition can only happen when
$f_{\xi,p}=(x+1)^2$, hence
$\left(2-\frac{\operatorname{Tr}(\xi)^2}{q}\right)$ must be equal to
$2$, i.e., $\Re(\xi)\equiv 0 \pmod{p}$.

The cases $\ell_{\xi,p} =4$ and $\ell_{\xi,p} =6$, which were treated
by A. Leite, in his Master's Thesis \cite{Antonio}, can also be easily
obtained from Theorem~\ref{mainthm}.

\begin{cor}[Propositions 3.3 and 3.4 of \cite{Antonio}]
  The non-trivial cycles of $\tau_{\xi,p}$ have length $4$ if and only
  if $2\N(\xi)\equiv \operatorname{Tr}(\xi)^2 \pmod{p}$, and they have
  length $6$ if and only if
  $3\N(\xi)\equiv \operatorname{Tr}(\xi)^2 \pmod{p}$.
\end{cor}

Set $\alpha = 2-\frac{\operatorname{Tr}(\xi)^2}{q}$ and rewrite
$f_{\xi,p}$ as $x^2+\alpha x+1$. As an example, consider the
particular case $\ell_{\xi,p}=4$. The previous corollary tell us that
this can only occur if
$2\N(\xi)\equiv \operatorname{Tr}(\xi)^2\pmod{p}$. This means that
$\alpha=0$, and so $f_{\xi,p}(x)=x^2+1=\Phi_4(x)$.  From Theorem
\ref{thm1}, the number of fixed points of $\tau_{\xi,p}$ is given by
$1+\left(\frac{\Re(\xi)^2-q}{p}\right)$, thus if the non-trivial
cycles have length $4$, we have that $2\Re(\xi)^2\equiv q \pmod{p}$,
and so the number of fixed points is
$1+\left(\frac{-\Re(\xi)^2}{p}\right)=1+\left(\frac{-1}{p}\right)$.
Therefore, when $\ell_{\xi,p}=4$ the number of fixed points does not
depend on the quadratic character of $\Re(\xi)^2-q$, but simply on the
quadratic character of $-1 \pmod{p}$.  This means that we have two
fixed points if $p\equiv 1 \pmod{4}$ and we do not have fixed points
if $ p\equiv 3 \pmod{4}$. Thus, for the primes $p$ of the form $4k+1$,
the permutation consists of $2$ fixed points and $(p-1)/4$ cycles of
length $4$, and for the primes of the form $4k+3$ the permutation
consists of $(p+1)/4$ cycles of length $4$.


Now suppose that given an odd prime $p$ we want to find a quaternion
$\xi$ such that $\ell_{\xi,p}=5$. First note that this can only happen
if $p=5$ or if $5\mid p\pm1$. Computing the resultant of $f_{\xi,p}$
and $\Phi_5$ one obtains $\alpha^2-\alpha-1\equiv 0 \pmod{p}$. Hence
$\alpha^2\equiv \alpha+1 \pmod{p}$. For example, for $p=19$ and
requiring that $\Re(\xi)=1$, by doing some calculations, our condition
is equivalent to $\N(\xi)\equiv7\pmod{19}$, and so we can use, for
example, the quaternion $\xi=1+2i+j+k$.

Consider now the situation where we are given a prime $p$ and a
quaternion $\xi$, and we want to find $\ell_{\xi,p}$. For example, let
again $p=19$ and $\xi=3-2i-2j$. We want to find $t>1$, such that
$(f_{\xi,p}, \Phi_t)\neq 1$. In order to do that, notice that, modulo
$19$, we have
\begin{align*}
f_{\xi,p}(x)=x^2+\left(2-\frac{36}{17}\right)x+1 =\Phi_3(x).
\end{align*}
Hence, $\ell_{\xi,19}=3$, which means the non-trivial cycles have
length precisely $3$.

An interesting case arises when $\ell_{\xi,p}$ is exactly equal to
$p$. In this case $\tau_{\xi,p}$ has only one fixed point and the rest
of primes lie in a cycle of length $p$. From our earlier discussion,
this can only happen when $\Re(\xi)^2\equiv \N(\xi)\pmod{p}$. We now
prove that, in this case, we can always find a quaternion $\xi$ in
these conditions. In fact, we will prove that we can always find two
different quaternions $\xi$ and $\xi'$ for which
$\ell_{\xi,p}=\ell_{\xi',p}=p$, but the fixed points are different.

\begin{prop}
  Given any odd prime $p$ we can always find a prime quaternion $\xi$
  such that $\ell_{\xi,p}=p$.
\end{prop}

 \begin{proof}
   Let $\xi \equiv a+bi+cj+dk \pmod{p}$ and
   $\operatorname{N}(\xi)\equiv q\pmod{p}$.  Following the above
   discussion, we know that for a $\xi$ to induce a permutation with
   only a fixed point and cycle of length $p$, we must have that
   $a^2 \equiv q \pmod{p} $, i.e., $q=a^2+pk$, for some $k \in \bbZ$.
   Hence, we need to show that for any odd prime $p$ we can always
   find $\xi =a+bi+cj+dk$ such that $\N(\xi)=a^2+pk$, for some
   $k \in \mathbb{Z}$. For that, we fix $a=2$ and choose
   $r \in \mathbb{N}$, with $r<8$ such that $p\equiv r \pmod{8}$. As
   $p$ is odd, we must have $r \in \{1,3,5,7\}$. Then, for any
   $k \in \bbN$,
   \[ (8k+r)p \equiv r^2\equiv 1 \pmod{8}. \] By Legendre's three
   squares theorem, this means that we can write the number $(8k+r)p$
   as a sum of three squares. Furthermore, by Dirichlet's Theorem, the
   sequence
   \[ 4+(8k+r)p = (4+pr) + 8pk \quad (k\in\bbN)\] has an infinitely
   many primes, since $4+pr$ is odd and not divisible by $p$ and so
   $(4+pr,8p)=1$. Thus, we can choose $k_0 \in \bbN$, so that
   $4+(8k_0+r)p$ is a prime number. Now, we know that we can write the
   number $(8k_0+r)p$ as a sum of three squares and, in fact, by
   \cite[Theorem 30.1.3]{Voight}, we can write as
   $(8k_0+r)p = b^2+c^2+d^2$, with $\gcd(b,c,d)=1$ and $b,c \neq
   0$. Therefore, the quaternion $\xi=2+bi+cj+dk$ fixes only one point
   of our permutation, the remainder belonging to a cycle of length
   $p$.
\end{proof}

\section{Fixed points of the metacommutation map}

As noted by Cohn and Kumar in \cite{CohnKumar}, if $\pi$ is a fixed
point of the metacommutation map by $\xi$, it is a left and right
divisor of $\pi\xi=\xi'\pi$. This motivates us to recall a result
presented in \cite{Common}, which gives the conditions under which a
quaternion in $\mc H$ has common left and right divisors.

Let $\alpha$ be a primitive quaternion in $\mc H$ whose norm is
divisible by a positive integer $m$. To know if $\alpha$ has any
common left and right divisor of norm $m$ is the same as to know if
exist $\beta, \gamma$ and $\gamma'$ in $\mc H$ such that $\N(\beta)=m$
and $\alpha=\beta \gamma =\gamma ' \beta$.  If we define the sets:

\begin{equation*}
  L_m(\alpha)=\{\beta \in \mc H \mid \alpha=\beta \gamma \text{ with }
  \gamma \in \mc H \text{ and } \N(\beta)=m \}, 
\end{equation*}
and
\begin{equation*}
  R_m(\alpha)=\{\beta \in \mc H \mid \alpha= \gamma ' \beta \text{
    with } \gamma ' \in \mc H \text{ and } \N(\beta)=m \}, 
\end{equation*}
as the sets of left and right factors of $\alpha$ with norm $m$, then
we want to determine
\begin{equation*}
  L_m(\alpha) \cap R_m(\alpha),
\end{equation*}
for a given $\alpha$ and $m \mid \N(\alpha)$. As in \cite{Common} we
make following distintion: a quaternion $\alpha\in \mc H$ is of
\textbf{type $(I)$} if $\alpha \in \mc L$. Otherwise is of
\textbf{type $(II)$}. For any element
$\alpha= a_0+a_1 i + a_2 j + a_3 k \in \mc H$, we set $a_i'=a_i$ if
$\alpha$ has type $(I)$ and $a_i'=2a_i$ if $\alpha$ is type $(II)$.

\begin{theorem}[Theorem 1 in \cite{Common}]\label{CommF}
  Suppose that $\alpha= a_0+a_1 i + a_2 j + a_3 k \in \mc H$ is
  primitive and let $m$ be an odd integer such that
  $m \mid \N(\alpha)$. Then,

\begin{equation*}
  L_m(\alpha) \cap R_m(\alpha) \neq \emptyset
\end{equation*}
if and only if there exists
$\beta= b_0+b_1 i + b_2 j + b_3 k \in \mc H$ such that $\N(\beta)=m$
and $a_i'b_j' \equiv a_j'b_i' \pmod{m}$, for all $i \neq j$. Moreover,
$\beta \in L_m(\alpha) \cap R_m(\alpha) $.
\end{theorem}
  
Now that we have established Theorem \ref{CommF}, we can look at it
from the point of view that most interest us. For instance, consider a
primitive quaternion $\alpha$ whose norm is $\N(\alpha)=pq$, where $p$
and $q$ are odd rational primes, and for which exists a factorizaion
factorization $\alpha=\pi\xi$, modeled on $pq$. Then, if $\pi$ is a
fixed point of map $\tau_{\xi,p}$ we have
\begin{equation*}
  \pi\in  L_m(\alpha) \cap R_m(\alpha).
\end{equation*}
Therefore, with the help of the Theorem \ref{CommF}, we can establish
the following corollary:
\begin{cor}
  Let $\alpha=a_0+a_1i+a_2j+a_3k$ be primitive such that
  $\N(\alpha)=pq$, with $p$ an odd prime. Then
  $\pi=b_0+b_1i+b_2j+b_3k \in \mc H$ is a fixed point of the
  metacomutation map if and only if $\N(\pi)=p$ and
  $a_i'b_j'\equiv a_j' b_i' \pmod{p}$, for all $i \neq j$.
\end{cor}

Recall that $\pi$ is a fixed point of the metacommutation map, then
this map sends the class of $\pi$ to itself, thus every left associate
is also a fixed point. Hence, to check if a point is fixed, we can
always choose a representative in its class that belongs to $\mc L$.
Let us start by dealing with the case $\xi\in \mc L$. Set
$\alpha=\pi\xi=a_0+a_1i+a_2j+a_3k\in \mc L$ with
$\pi=b_0+b_1i+b_2j+b_3k \in \mc L$. If $\pi$ is a fixed point of
$\tau_{Q,p}$, then $\alpha=\pi\xi =\xi'\pi$, for some
$\xi' \in \mc H$, and so we have the system of congruences
\begin{equation*}
  \begin{cases}
    a_0b_1 \equiv a_1b_0\\
    a_0b_2 \equiv a_2b_0\\
    a_0b_3 \equiv a_3b_0\\
    a_1b_2 \equiv a_2b_1\\
    a_1b_3 \equiv a_3b_1\\
    a_2b_3 \equiv a_3b_2\\
  \end{cases}\pmod{p}
\end{equation*}
The first three congruences can be condensed into the unique
congruence
\begin{equation*}
  a_0 \mathcal{V}(\pi) \equiv  b_0 \mathcal{V}(\alpha) \pmod{p},
\end{equation*}
and the last three congruences are simply equivalent to
\begin{equation*}
  \mathcal{V}(\alpha) \times \mathcal{V}(\pi) \equiv 0 \pmod{p}.
\end{equation*}
We have that $p \nmid a_0b_0$, because if $p \mid a_0 b_0$, since
$\N(\pi)=p$, and $p$ cannot divide $b_0$, $p$ must divide
$a_0$. However, $a_0=\pi \cdot \bar{\xi}$, thus we would have that
$p \mid \xi$, which cannot happen.  Therefore, the first congruence
implies the second one. So, modulo $p$, the vector parts of $\alpha$
and $\pi$ are colinear. Notice that we can write the first congruence
of our system as,
\begin{align*}
  (\pi \cdot \bar{\xi})b_1- b_0 (\pi\cdot i \bar{\xi})
  &\equiv 0 \pmod{p} \Leftrightarrow \\ 
  \pi\cdot (b_1-b_0i)\bar{\xi }
  &=0 \Leftrightarrow  \Re((b_1+b_0i)\pi \xi )\equiv0 \pmod{p}.
\end{align*}
Similarly, for the other two congruences follows that,
\begin{equation*}
  \pi \cdot (b_2-b_0j)\bar{\xi}=0 \Leftrightarrow \Re((b_1+b_0j)\pi\xi
  )\equiv0 \pmod{p}, 
\end{equation*}
\begin{equation*}
  \pi\cdot (b_3-b_0k)\bar{\xi}=0 \Leftrightarrow
  \Re((b_1+b_0k)\pi\xi)\equiv0 \pmod{p}. 
\end{equation*}
In our case, the last three equations of our system are implied by the
first three, thus, from the above calculations, we have:
\begin{prop}\label{propfim}
  Given a Hurwitz prime $\pi =b_0+b_1i+b_2j+d_3k\in \mc L$ of odd norm
  $p$, and a prime quaternion $\xi \in \mc L $ such that $\N(\xi )=q$,
  then $\pi$ is a fixed point of $\tau_{\xi ,p}$ if and only if
  $\Re(\delta_i\pi\xi)\equiv 0 \pmod{p}$, where $\delta_1=b_1+b_0i$,
  $\delta_2=b_2+b_0j$, and $\delta_3=b_3+b_0k$.
\end{prop}

When $\xi \in \mc H\setminus\mc L$, we only need to change the
conditions to $2\Re(\delta_i\pi\xi)\equiv0$, for each $i$.

\begin{prop}
  Given an odd prime $p$, we can always find two prime quaternions
  $\xi$ and $\xi'$, such that $\ell_{\xi,p}=\ell_{\xi',p}=p$ but in
  which the fixed points are different.
\end{prop}

\begin{proof}
  We already now that we can find a quaternion $\xi=2+bi+cj+dk$ that
  fixes only one fixed point, say $\pi=b_0+b_1i+b_2j+b_3k$, which we
  assume to be in $\mc L$.  Start by considering
  $\xi'=2-bi+cj+dk=\rho-2bi$. Since $\pi$ is a fixed point of
  $\tau_{\xi,p}$, we know that we must have from Proposition
  \ref{propfim}
  \begin{align}\label{fixed_point}
    \delta_1\pi \xi \cdot 1\equiv 0 \pmod{p} \\
    \delta_2\pi \xi \cdot 1\equiv 0 \pmod{p} \\
    \delta_3\pi \xi \cdot 1\equiv 0 \pmod{p} 
  \end{align}
  where $\delta_1 = b_1+b_0i$, $\delta_2 = b_2 + b_0j$ and
  $\delta_3 = b_3 + b_0k$.  Let us start by considering the case
  $p \equiv 3 \pmod{4}$.  If the permutation induced by $\xi'$ also
  fixes $\pi$, then one should also have
  \begin{equation*}
    \delta_1\pi\xi'\cdot 1\equiv 0 \pmod{p} \iff \delta_1\pi (\xi
    -2bi)\cdot 1 \equiv 0 \pmod{p}. 
  \end{equation*}
  Putting the two identities together, this is simply
  \begin{equation*}
    2b\delta_1\pi\cdot i \equiv 0 \pmod{p} \iff 2b(b_0^2+b_1^2)\equiv 0\pmod{p}.
  \end{equation*}
  Since $\gcd(b,c,d)=1$, we can assume that $p\nmid b$. As $p$ is the
  norm of $\pi$, we must have either $b_0^2+b_1^2=p$ or
  $b_0^2+b_1^2=0$, that is, $\pi=b_0+b_1i$ or $\pi=b_2j+b_3k$, but
  this means, in either case, that $p$ is the sum of two squares, and
  so by Fermat's two squares theorem, $p\equiv 1 \pmod{4}$, which is
  absurd.
    
  Now consider the case $p \equiv 1 \pmod{4}$. Using identity
  \eqref{fixed_point}, we have $\pi=b_0+b_1 i$ or
  $\pi=b_2j+b_3k$. Note that if we are in the second case, then we
  have $j\pi=-b_2+b_3i$, and can replace $\pi$ by an element of its
  associate class. Hence, essentially it suffices to deal with the
  case when $\pi=b_0+b_1i$.  Now, we will consider
  $\xi''=\rho-2cj$. Thus we must have
  \[ \delta_2\pi\xi'' \cdot 1\equiv 0 \pmod{p} \iff \delta_2\pi (\rho
    -2cj)\cdot 1 \equiv 0 \pmod{p}. \] This implies that
  \begin{equation*}
    2c\delta_2\pi \cdot j \equiv 0 \pmod{p} \iff 2c b_0^2 \equiv 0 \pmod{p},
  \end{equation*}
  and, if assume that $p\nmid c$, we have either $b_0^2=0$ (and so
  $b_1^2=p$) or $b_0^2=p$. But either case is absurd, thus $\xi''$
  does not fix $\pi$. Lastly if $p\mid c$ we consider
  $\xi'''=2-ci-bj+dk=\rho +(-b-c)i+(-b-c)j$ and simplyfing we have
  \[\delta_2\pi \xi''' \cdot 1 \equiv 0 \pmod{p} \iff -b\delta_2\pi
    (i+j) \cdot 1 \equiv 0. \] Hence, we have
  \[ b\delta_2\pi \cdot i + b\delta_2\pi \cdot j \equiv 0 \pmod{p}
    \iff 0 + b_0^2\equiv 0 \pmod{p}, \] which again is absurd.
  Therefore, in every scenario, we can find two different
  permutations, one that fixes only $\pi$ and another that fixes a
  different point, therefore we have our result.
\end{proof}
\bibliographystyle{plainnat}
\bibliography{MetaCycleStructure_08042025}

\begin{thebibliography}{9}
\providecommand{\natexlab}[1]{#1}
\providecommand{\url}[1]{\texttt{#1}}
\expandafter\ifx\csname urlstyle\endcsname\relax
  \providecommand{\doi}[1]{doi: #1}\else
  \providecommand{\doi}{doi: \begingroup \urlstyle{rm}\Url}\fi

\bibitem[Abouzaid et~al.(2013)Abouzaid, Alper, DiMauro, Grosslight, and
  Smith]{Common}
Mohammed Abouzaid, Jarod Alper, Steve DiMauro, Justin Grosslight, and Derek
  Smith.
\newblock Common left- and right-hand divisors of a quaternion integer.
\newblock \emph{Journal of Pure and Applied Algebra}, 217\penalty0
  (5):\penalty0 779--785, 2013.

\bibitem[Cohn and Kumar(2015)]{CohnKumar}
Henry Cohn and Abhinav Kumar.
\newblock {Metacommutation of Hurwitz Primes}.
\newblock \emph{Proceedings of the American Mathematical Society}, 143\penalty0
  (4):\penalty0 1459--1469, 2015.

\bibitem[Conway and Smith(2003)]{ConwaySmith}
John~H. Conway and Derek~A. Smith.
\newblock \emph{{On Quaternions and Octonions}}.
\newblock AK Peters/CRC Press, 2003.

\bibitem[Forsyth et~al.(2016)Forsyth, Gurev, and Shrima]{ForsythEtAll}
A.~Forsyth, J.~Gurev, and S.~Shrima.
\newblock {Metacommutation as a Group Action on the Projective Line over
  $\mathbb{F}_p$}.
\newblock \emph{Proceedings of the American Mathematical Society}, 144\penalty0
  (11):\penalty0 4583--4590, 2016.

\bibitem[Leite(2024)]{Antonio}
A.~Leite.
\newblock The metacommutation problem in the hurwitz integers.
\newblock Master's thesis, Faculdade de Ciências da Universidade do Porto,
  2024.

\bibitem[Lidl and Niederreiter(1996)]{LidlNiederreiter}
Rudolf Lidl and Harald Niederreiter.
\newblock \emph{Finite Fields}.
\newblock Cambridge University Press, 1996.

\bibitem[Oswald and Steuding(2023)]{hurwitz2023}
Nicola Oswald and J\"orn Steuding.
\newblock \emph{Hurwitz’s Lectures on the Number Theory of Quaternions}.
\newblock EMS Press, 2023.

\bibitem[Tsopanidis(2020)]{Nikos}
Nikolaos Tsopanidis.
\newblock \emph{The Hurwitz and Lipchitz Integers and Some Applications}.
\newblock PhD thesis, Universidade do Porto, 2020.

\bibitem[Voight(2021)]{Voight}
Hojn Voight.
\newblock \emph{Quaternion Algebras}.
\newblock Springer, 2021.

\end{thebibliography}
\end{document}